%% file: main.tex
\documentclass[a4paper]{amsart}

\usepackage[utf8]{inputenc}
\usepackage{tikz}
\usepackage{tikz-cd} 
\usetikzlibrary{babel}
\usepackage{amssymb} 
\usepackage{ stmaryrd } 
\usepackage{amsmath} 
\usepackage{amsthm} 
\usepackage{verbatim}
\usepackage{hyperref}
\usepackage{cleveref} 
\usepackage{float}

\usepackage[backend=biber, style=alphabetic, sorting=ynt]{biblatex}
\usepackage{csquotes}
\newtheorem{theo}{Theorem}

\newtheorem{cor}{Corollary}
\numberwithin{cor}{section}

\newtheorem{prop}{Proposition}
\numberwithin{prop}{section}

\newtheorem{lemma}{Lemma}

\theoremstyle{definition}
\newtheorem{defi}{Definition}

\theoremstyle{remark}
\newtheorem{rmk}{Remark}

\DeclareMathOperator{\Spec}{Spec}

\DeclareMathOperator{\Sing}{Sing}

\title{On the Ramanujan Vector Field modulo \texorpdfstring{$p$}{p}}

\author{Frederico Bianchini}

\address{Instituto Nacional de Matemática Pura e Aplicada\\
Estrada Dona Castorina, 110\\
Rio de Janeiro-RJ, Brazil}

\email{frederico.bianchini@impa.br}

\begin{document}

\begin{abstract}
	For every prime \(p \geq 5\), we compute the \(p\)-th power of the Ramanujan vector field that arises from the differential relations discovered by Ramanujan for the Eisenstein series \(E_2,E_4\) and \(E_6\). Our method results in explicit equations for the \(p\)-th power and uses classical results of Serre and Swinnerton-Dyer about modular forms modulo \(p\). From this, we verify that a general conjecture by Sheperd-Barron and Ekedahl is valid for the Ramanujan vector field. Furthermore, we consider the affine realization of a certain moduli space of elliptic curves where the Ramanujan vector field is defined, and describe -- in characteristic \(p\) -- the locus given by supersingular elliptic curves in two ways: a classical one -- using equations for the supersingular polynomial -- and a new one as the singular set of some vector fields. Additionally, we prove that the Ramanujan vector field is transversal to this locus. 
\end{abstract} 

\maketitle

\input{introduction.tex}

\input{the-p-th-power.tex}

\input{supersingular_locus}

\printbibliography
\end{document}

%% file: introduction.tex
\section{Introduction}

    The classical differential relations between the normalized Eisenstein series 
    \begin{align*}
        E_2 &= 1 - 24 \sum_{n=0}^{\infty} \sigma_1(n) q^n \\
        E_4 &= 1 + 240 \sum_{n=0}^{\infty} \sigma_3(n) q^n \\ 
        E_6 &= 1 - 504 \sum_{n=0}^{\infty} \sigma_5(n) q^n 
    \end{align*}
    found by Ramanujan (\cite{ramanujan-on-certain-arithmetical-functions}) are given by 
    \begin{equation}
        \label{Ramanujan_system} 
        \theta E_2 = \frac{E_2^2 - E_4}{12} \quad \theta E_4 = \frac{E_2E_4 - E_6}{3} \quad \theta E_6 = \frac{E_2 E_6 - E_4^2}{2}
    \end{equation}
    where \(\theta := q \frac{\partial}{\partial q}\). Associated to these equations, we consider the \emph{Ramanujan vector field} 
    \begin{equation}
        \label{Ramanujan_vf}
        R:= \frac{e_2^2 - e_4}{12} 
   	\frac{\partial}{\partial e_2} + 
   	\frac{e_2e_4 - e_6}{3}\frac{\partial}{\partial e_4} + 
	\frac{e_2 e_6 - e_4^2}{2} 
	\frac{\partial}{\partial e_6}             
    \end{equation}
    defined globally in the 3-dimensional affine space with coordinates \((e_2,e_4,e_6)\). The analytic trajectories of the vector field \eqref{Ramanujan_vf} in \(\mathbf{C}^3\) correspond to local solutions of the differential system defined by \eqref{Ramanujan_system} and, considering all the solutions together, we can define a holomorphic foliation \(\mathcal{F}(R)\) in \(\mathbf{C}^3\). This foliation has a singular set given by the twisted cubic \(\{(s, s^2, s^3): s \in \mathbb{C}\}\) which is contained in the \(R\)-invariant algebraic hypersurface defined by the zeroes of \(e_4^3 - e_6^2\). It is known that the algebraic leaves of \(\mathcal{F}(R)\) are exactly the ones contained in \(\{ e_4^3 - e_6^2= 0\}\) (see \cite{Nesterenko_1996}, \textsection 4 and \cite{movasati-on-elliptic-modular-foliations}, Theorem 1). In other words, the foliation \(\mathcal{F}(R)\) restricted to \(\mathbf{C}^3 \setminus \{e_4^3 - e_6^2 = 0\}\) \emph{ has only transcendental leaves}. These properties were first explored in Nesterenko's proof that, for each \(q \in \{z \in \mathbf{C}: 0 < |z| <1\}\), the transcendence degree of \(\mathbf{Q}(q, E_2(q), E_4(q), E_6(q))\) is greater or equal than three (see \cite{Nesterenko_1996}).
    
    Looking at the definition of \(R\), it is clear that it defines a vector field in the affine scheme \(\mathbf{A}^3_{\mathbf{Z}[1/6]}\), thus reducing the Ramanujan vector field modulo \(p\) makes sense for every prime \(p \geq 5\). A general conjecture put forward by N. Sheperd-Barron and T. Ekedahl (see \cite{bost}, page 165) predicts that the lack of algebraic leaves for \(R\) in the complex analytic space \(\mathbf{C} \setminus \{e_4^3 - e_6^2 = 0\}\) implies that, for an infinite number of primes \(p \geq 5\), the foliation induced by \(R\) in the affine scheme 
    \[
        U_{\mathbf{F}_p} := \Spec \mathbf{F}_p[e_2,e_4,e_6, (e_4^3 - e_6^2)^{-1}] \subset \mathbf{A}^3_{\mathbf{F}_p}
    \] 
    must have \(p\)-curvature different from zero. In simple terms, this means that the \(p\)-th power \(R^p\) is not a multiple of \(R\) by a regular function of \(U_{\mathbf{F}_p} \). \emph{We verify the prediction for this specific case by explicitly computing \(R^p\).} In fact we show a little bit more, namely that the \(p\)-curvature is non zero for \emph{every} prime \(p \geq 5\) (see \cref{cor: Sheperd-Barron-Ekedahl}). 

    In order to do that, we use the fact that the tangent sheaf \(\mathcal{T}_{U_{\mathbf{F}_p}}\) is generated as an \(\mathcal{O}_{U_{\mathbf{F}_p}}\)- module by the three vector fields 
    \begin{align}
        \label{eq:basis}
        \begin{split}
        R &:= \text{Ramanujan vector field} \\
        F &:= - 12 \frac{\partial}{\partial e_2} \\
        H &:= 2 e_2 \frac{\partial}{\partial e_2} + 4 e_4
			\frac{\partial}{\partial e_4} + 6 e_6 
			\frac{\partial}{\partial e_6}
        \end{split}
    \end{align}
    and that this basis forms an \(\mathfrak{sl}_2\)-triple. These last two properties are crucial for the techniques we used and they have been studied from different perspectives (for example, \cite{guillot} explores the same properties for the Halphen family of differential equations and \cite{bogo-younes} develops a classification of algebras with similar structure). Furthermrore, we use classical results on modular forms modulo \(p\) obtained by Serre and Swinnerton-Dyer (\cite{swd-modl}, \cite{serre-congruences-et-formes-modulaires}). The precise formulation of the first result is as follows.
	
	\begin{theo}
		\label{thm:main_theorem}
		Let \(k\) be a field of characteristic \(p \geq 5\). 
		Then, the \(p\)-th power of the Ramanujan vector field as a global section of the tangent sheaf \(\mathcal{T}_{\mathbf{A}^3_{k}}\) can be written as 
		\begin{equation}
            \label{eq:p-th_power}
            R^p =  \frac{\tilde B^2 - e_4 \tilde A^2}{12} \frac{\partial}{\partial e_2} + \frac{\tilde A(e_4 \tilde B - e_6 \tilde A)}{3} \frac{\partial}{\partial e_4} + \frac{\tilde A(e_6 \tilde B - e_4^2 \tilde A)}{2} \frac{\partial}{\partial e_6}
		\end{equation}
        where \( \tilde{A},  \tilde{B}\) are the reduction modulo \(p\) of the unique polynomials \(A,B \in \mathbf{Q}[e_4,e_6]\) such that 
        \begin{align*}
            A (E_4,E_6) &= E_{p-1} \\
            B  (E_4,E_6) &= E_{p+1}
        \end{align*}
        Here \(E_{\nu}\) denotes the \(\nu\)-th normalized Eisenstein series.
	\end{theo}	
    
    Next, the affine scheme 
    \begin{equation}
        \label{eq:U}
                U := \Spec \mathbf{Z}[1/6, e_2,e_4,e_6,(e_4^3 - e_6^2)^{-1}]
    \end{equation}
    represents a certain moduli problem of elliptic curves over \(\mathbf{Z}[1/6]\) (see \cref{subsec:moduli_space} for a brief explanation and \cite{movasati-2012}, \cite{tiago_higher_ramanujan_equations} for a more detailed account). Given an algebraically closed field \(k\) of characteristic \(p \geq 5\), a natural question is to consider the locus given by classes of supersingular elliptic curves in \(U_k = U \otimes k\) and how it interacts with the Ramanujan vector field. We show that the Ramanujan vector field is transversal to this locus. In order to do that, we use a well-known representation of the so-called \emph{supersingular polynomial}
    \[
       ss_p(t) := \prod_{E/k \text{ supersingular}} (t- j(E)),
    \]
    in terms of the polynomial \(\tilde A\) of \cref{thm:main_theorem} (see \cite{kaneko-zagier-supersingular_j-invariants}). Then, denoting 
    \[
        J:  U_k \mapsto \mathbf{A}^1_k
    \]
    the morphism of \(k\)-schemes given by the \(j\)-invariant we consider the \emph{supersingular locus} as the reduced scheme\footnote{The non-reduced scheme has multiple components (see \cref{prop:unreduced_scheme}) and we expect that the interpretation of this behavior might become clear when doing the same process in the setting of moduli stacks.}\ 
    \begin{equation}
        \label{eq:supersingular_locus}
        SS_p := J^{-1}(V(ss_p))_{\text{red}} \subset U_k
    \end{equation}
    where \(V(ss_p) \subset \mathbf{A}^1_k\) is the zero set of \(ss_p\). Lastly, we use \cref{thm:main_theorem} to give another description of the supersingular locus as the singular set of the vector field \(R^p + \left(\frac{\tilde B}{12}\right)^2 \cdot F\). The statements are as follows.

\begin{theo}
    \label{thm:supersingular_locus}
    Let \(k\) be an algebraically closed field of characteristic \(p\geq 5\).  
    \begin{enumerate}
        \item The supersingular locus \(SS_p \subset U_k\), as defined in \cref{eq:supersingular_locus}, is the zero set of \(\tilde A(e_4,e_6)\);
        \item \(SS_p\) has \(n_p := \# \{\text{supersingular } j \text{-values in } k\}\) irreducible components, disjoint from each other and each component is a smooth surface;
        \item \(SS_p\) is the singular set of the vector field \(R^p + \left(\frac{\widetilde B}{12}\right)^2 \cdot F\);
        \item The Ramanujan vector field is transversal to \(SS_p\) in every point, that is, \(SS_p\) has no \(R\)-invariant subvarieties.
    \end{enumerate}
\end{theo}

   Another possible approach to obtain information about the \(p\)-th power \(R^p\) is the geometric one. For example, if one considers the family \(\mathcal{E}/U\) of elliptic curves described in \cref{eq:univ_family}, it is known (see \cite{movasati-2012}) that -- with the choice of a suitable basis of \(H^1_{dR}(\mathcal{E}/U)\) -- the Ramanujan vector field satisfies a simple linear equation with respect to the Gauss-Manin connection \(\nabla\) in \(H^1_{dR}(\mathcal{E}/U)\). From this, one may try to apply general theorems (for instance, Katz's theorems on the nilpotence of the Gauss-Manin connection \cite{katz_nilpotent} or the solution of Grothendieck-Katz conjecture for Picard-Fuchs equations \cite{katz-72}) to this case in order to obtain information about the \(p\)-curvature. An approach like this would be interesting since it may allow generalizations for foliations with similar structure defined over certain moduli spaces of abelian varieties (see \cite{tiago_higher_ramanujan_equations}) or even Calabi-Yau manifolds (see the motivational part of the introduction of \cite{cao-movasati-villaflor} and references therein).

    We chose not to follow this path since the explicit nature of our methods resulted in very precise formulas that work for every prime \(p \geq 5\). Besides the formula for \(R^p\), our computations allowed us to obtain, as a by-product, explicit factorizations of the polynomials \(\tilde A\) and \(\tilde B\) (see \cref{lemma:factorization}). Furthermore, the techniques used are elementary in nature and based in results that are known for a very long time.

    \subsection*{Aknowledgements}

I gladly acknowledge that this work grew out of a number of fruitful discussions with my PhD advisor, Hossein Movasati. I also had many discussions with Tiago Fonseca, who carefully revised parts of the text. Felipe Espreafico and Roberto Villaflor helped in moments when I was stuck. I leave my heartfelt thanks to them. I would also like to thank the staff at IMPA for the wonderful environment they provide. The present work was carried out with the support of Conselho Nacional de Desenvolvimento Cient\'ifico e Tecnológico of Brazil (CNPq) and partially developed at the Institut des Hautes Études Scientifiques (IHES) in France, where I was visiting for a month in 2024.

%% file: the-p-th-power.tex
\section{The \texorpdfstring{$p$}{p}-th power of the Ramanujan vector field} 

    \subsection{Modular and quasi-modular forms} 
   
    We briefly review some results and definitions about modular and quasi-modular
    forms for \(SL_2(\mathbf{Z})\) and their characteristic \(p\) counterparts. In what follows, given \(S\) a \(\mathbf{Z}[1/6]\)-algebra, we consider in \(S[e_2,e_4,e_6,(e_4^3-e_6^2)^{-1}]\) the graded structure induced by assigning 
    \begin{align*}
        \deg e_2 &= 2 \\
        \deg e_4 &= 4 \\
        \deg e_6 &= 6.
    \end{align*}
    Thus, by a \emph{homogeneous polynomial of degree \(d\)} we mean a polynomial that has every monomial of degree \(d\) according to the gradation described above.  
    
    \subsubsection{Eisenstein series and modular forms}
    
    For each positive even integer \(\nu\) let \(E_{\nu}\) denote the \(\nu\)-th normalized Eisenstein series 
    \[
		E_{\nu}(\tau) := 1 - \frac{2\nu}{B_{\nu}} \sum_{n=1}^{\infty} \sigma_{\nu-1}(n) q^n, \,\, \tau \in \mathbf{H}
	\]	
    where \(B_{\nu}\) is the \({\nu}\)-th Bernoulli number, \(\sigma_{\mu}(n) := \sum_{d \mid n} d^{\mu}\), \(q(\tau) := e^{2\pi i \tau}\) and \(\mathbf{H}\) is the Poincaré upper half-plane. For any \(\nu \geq 4\), the Eisenstein series \(E_\nu\) is a modular form of weight \(\nu\) for \(SL_2(\mathbf{Z})\) and \(E_2\) is a quasi-modular form for the same group. 

    The graded ring \(\mathbf{C}[e_4,e_6]\) is isomorphic to the \(\mathbf{C}\)-algebra \(\mathcal{M}(\mathbf{C})\) of modular forms for \(SL_2(\mathbf{Z})\) with grading given by weight. This isomorphism is given by \(e_i \mapsto E_i, i = 4,6\). Similarly, assigning degree \(2\) to the additional variable \(e_2\), the graded ring \(\mathbf{C}[e_2,e_4,e_6]\) is isomorphic to the \(\mathbf{C}\)-algebra \(\mathcal{QM}(\mathbf{C})\) of quasi-modular forms for \(SL_2(\mathbf{Z})\) (see \cite{kaneko-zagier-jacobi_theta_function}) where \(e_2 \mapsto E_2\). Lastly, consider the degree 12 polynomial  
    \[
        e_4^3 - e_6^2 \in \mathbf{C}[e_4,e_6]
    \]
    which corresponds to a multiple of the modular discriminant. Then, the \(\mathbf{C}\)-algebra \(\mathbf{C}[e_4,e_6, (e_4^3 - e_6^2)^{-1}]\) is isomorphic to the algebra of \emph{weakly holomorphic modular forms}, i.e., those modular forms might have a pole at infinity. Lastly, the algebra \(\mathbf{C}[e_2,e_4,e_6, (e_4^3 - e_6^2)^{-1}]\) corresponds to \emph{weakly holomorphic quasi-modular forms}. We remark that this is exactly the ring of regular functions of \(U_\mathbf{C}\) (see \cref{eq:U}).

    \subsubsection{The Ramanujan-Serre derivative}
    \label{subsec:Ramanujan-Serre}
    In the following, we only consider modular and quasi-modular forms for \(SL_2(\mathbf{Z})\). For each even integer \(\nu \geq 4\), given a modular form \(f\) of weight \(\nu\), the operation
    \begin{equation}
        \label{eq:Ramanujan-Serre_derivative}
        \partial_\nu(f) := 12\theta f - \nu E_2 f
    \end{equation}
    is called the \emph{Ramanujan-Serre derivative} of \(f\) and it defines a homogeneous derivation of weight \(2\), meaning that \(\partial_\nu(f)\) is a modular form of weight \(\nu + 2\). Here, \(\theta \) denotes the derivation 
    \[
        q \frac{\partial}{\partial q} \left( \sum_{n \geq s}^{\infty} a_n q^n \right) = \sum_{n \geq s}^{\infty} n a_n q^n
    \]
    defined for any formal Laurent series.

    We make a few observations about the Ramanujan-Serre derivative and \(\theta\).  Firstly, the fact that the triple \((E_2,E_4,E_6)\) is a solution to the system \eqref{Ramanujan_system} is equivalent to the fact that the isomorphism 
    \begin{equation}
        \label{eq:isomorphism/C}
        \varphi: \mathbf{C}[e_2,e_4,e_6] \stackrel{\sim}{\longrightarrow} \mathcal{QM}(\mathbf{C}), \,\,\,\,\ e_2,e_4,e_6 \mapsto E_2,E_4,E_6
    \end{equation}
    fits in the commutative diagram
    \begin{equation} 
        \label{eq:commutative_diagram/C}
        \begin{tikzcd}
	       {\mathbf{C}[e_2,e_4,e_6] } & {\mathcal{QM}(\mathbf{C})} \\
	       {\mathbf{C}[e_2,e_4,e_6] } & {\mathcal{QM}(\mathbf{C})}
	       \arrow["\sim", from=1-1, to=1-2]
	       \arrow["R"', from=1-1, to=2-1]
	       \arrow["\theta", from=1-2, to=2-2]
	       \arrow["\sim"', from=2-1, to=2-2]
        \end{tikzcd}.
    \end{equation}
    This also corresponds to the curve \(\tau \mapsto (E_2(\tau), E_4(\tau), E_6(\tau))\) in \(\mathbf{C}^3\) being a leaf of \(\mathcal{F}(R)\) through the point \((1,1,1)\).

    The second and third observations are that the commutativity of diagram \eqref{eq:commutative_diagram/C} is preserved ``over \(\mathbf{Z}[1/6]\)" and that it also works in the weakly holomorphic setting. More precisely, for any \(\mathbf{Z}[1/6]\)-algebra \(S\), the \(S\)-algebra homomorphism
    \begin{equation}
        \label{eq:isomorphism/S}
        \varphi_S: S[e_2,e_4,e_6, (e_4^3 - e_6^2)^{-1}] \longrightarrow S((q)), \,\,\,\,\ e_2,e_4,e_6 \mapsto E_2,E_4,E_6
    \end{equation}
    is well defined and fits in the commutative diagram
    \begin{equation} 
        \label{eq:commutative_diagram/S}
        \begin{tikzcd}
	       {S[e_2,e_4,e_6,(e_4^3 - e_6^2)^{-1}] } & S((q)) \\
	       {S[e_2,e_4,e_6,(e_4^3 - e_6^2)^{-1}] } & S((q))
	       \arrow[from=1-1, to=1-2]
	       \arrow["R"', from=1-1, to=2-1]
	       \arrow["\theta", from=1-2, to=2-2]
	       \arrow[from=2-1, to=2-2]
        \end{tikzcd}
        .
    \end{equation}
    Here, \(S((q))\) denotes the ring of formal Laurent series with coefficients in \(S\). This fact is used in the proof of \cref{thm:main_theorem} (see \cref{eq:formal_solution}). 
    
    Lastly, we state the Ramanujan-Serre derivative version of \eqref{eq:commutative_diagram/S}. Consider the homomorphism 
     \begin{equation}
        \label{eq:isomorphism/2S}
        \hat \varphi_S: S[e_4,e_6] \longrightarrow S[[q]], \,\,\,\,\ e_4,e_6 \mapsto E_4, E_6.
    \end{equation}
    Then, for any homogeneous polynomial \(P(e_4,e_6)\) of degree \(\nu\) we have 
    \[\partial_\nu \circ \hat \varphi_S(P) \in S[[q]],\]
    since \(E_2\) has coefficients in \(\mathbf{Z}\). Moreover, we have the identity 
      \begin{equation} 
        \label{eq:commutative_diagram/Ram-Serre}
        \partial_\nu \circ \hat \varphi_S (P) = \hat \varphi_S \circ \partial(P)
    \end{equation}
    where 
    \begin{equation}
        \label{eq:Ramanujan-Serre_vf}
        \partial := - 4 e_6 \frac{\partial}{\partial e_4} - 6 e_4^2 \frac{\partial}{\partial e_6}
    \end{equation}
    is a derivation of \(S[e_4,e_6]\). We call \(\partial\) the \emph{Ramanujan-Serre vector field}. 

    \begin{rmk}
        The Ramanujan vector field can be written in the polynomial algebra \(S[e_2,e_4,e_6]\) as the derivation 
        \begin{equation}
            \label{eq:ramanujan_vf_2}
            R = \frac{1}{12}(\partial + e_2 H) 
        \end{equation}
        where \(H\) is the Euler vector field of this algebra as defined in \cref{eq:basis}.
    \end{rmk}
    
    \subsubsection{Modular forms modulo \texorpdfstring{$p$}{p}}   
    \label{sec:mod_forms_mod_p}
    
    We outline a small part of what was developed by Swinnerton-Dyer and Serre regarding the reduction of modular forms modulo primes (see \cite{swd-modl} and \cite{serre-congruences-et-formes-modulaires}). We reproduce the proof of Theorem 2 of \cite{swd-modl} for completeness and since it illustrates the arithmetical applications of the fact that the polynomials \(\tilde A\) and \(\tilde B\) are satisfy the differential relations 
    \begin{equation*}
        \begin{cases}
            \partial \tilde A = \tilde B \\
            \partial \tilde B = -e_4 \tilde A
        \end{cases}
    \end{equation*}
    with respect to the vector field \cref{eq:Ramanujan-Serre_vf}. In the end of the section we show that this implies the fact that the Ramanujan vector field has \( \tilde B - e_2  \tilde A\) as a first integral modulo \(p\). This is interesting since it contrasts with the fact that, in characteristic 0, \(R\) has \(\{e_4^3 - e_6^2 = 0\}\) as the only invariant hypersurface. We believe that \(\{e_4^3 - e_6^2 = 0\}\) and \( \{\tilde B - e_2  \tilde A = 0\} \) are the only \(R\)-invariant hypersurfaces modulo \(p\). Lastly, we would like to acknowledge that the fact that \(\tilde B - e_2 \tilde A\) is a first integral modulo \(p\) for \(R\) was earlier discussed by Movasati in \cite{movasati2025leafschemeshodgeloci}, where we learned this from.
    
    In what follows, we denote by \(p \geq 5\) a prime number, \(\mathbf{Z}_{(p)}\) the ring of \(p\)-integral rational numbers (rational numbers with denominator coprime to \(p\)) and \(\Delta := E_4^3 - E_6^2/1728\) the modular discriminant. Furthermore, we denote with a tilde the reduction modulo \(p\) of a polynomial with \(p\)-integral coefficients, that is, if \(P(e_4,e_6)\) is a polynomial with \(p\)-integral coefficients, \(\tilde P(e_4,e_6)\) denotes its image in \(\mathbf{F}_p[e_4,e_6]\).

    Let \(f = \sum_{n=0}^{\infty} a_n q^n\) be a modular form of weight \(\nu\) and consider \(G := \mathbf{Z} a_0 + \mathbf{Z} a_1 + \cdots \) the additive subgroup of \(\mathbf{C}\) generated by the coefficients of \(f\). Choose positive integers \(i,j\) such that \(\nu = 4i + 6j\) and consider 
    \[
        g := f - a_0 E_4^i E_6^j.
    \]
    There are two possiblities for \(g\): either \(g = 0\) (this happens when \(\nu < 12\)) or \(g\) is a cusp form of weigth \(\nu \geq 12\). In the second case, we can write \(g = \Delta \hat g\) for some modular form \(\hat g\) of weight \(\nu-12\). Thus, either \(f = a_0 E_4^i E_6^j\) or \( f = a_0 E_4^i E_6^j + \Delta \hat g\). Since \(f\) and \(a_0 E_4^i E_6^j\) have coefficients in \(G\) and \(\Delta\) has coefficients in \(\mathbf{Z}\), it is easy to show that \(\hat g\) has coefficients in \(G\). By induction on the weight of \(f\), we conclude that \emph{every modular form of pure weight is a polynomial in \(E_4, E_6\) and \(\Delta\) with coefficients in the additive group generated by its coefficients}. Moreover, since \(\Delta = \frac{E_4^3 - E_6^2}{1728}\) and \(1728 = 2^6 3^3\), if \(f\) has rational coefficients that are \(p\)-integral (recall that \(p \geq 5\)) then \(f\) is a polynomial in \(E_4\) and \(E_6\) with coefficients in \(G\). 

    Therefore, we have a bijection between the polynomial ring \(\mathbf{Z}_{(p)}[e_4,e_6]\) and the \(\mathbf{Z}_{(p)}\)-algebra consisting of all modular forms with \(p\)-integral \(q\)-expansion.
    
    \begin{defi}[\cite{swd-modl}, \textsection 3] The ring of \emph{modular forms modulo \(p\)} is the image of the composition  
    \[
       \mathbf{Z}_{(p)}[e_4,e_6] \longrightarrow \mathbf{Z}_{(p)}[[q]]  \longrightarrow \mathbf{F}_p[[q]],
    \]
    where the first map is given by \(e_i \mapsto E_i\), \(i = 4,6\) and the second one is given by reduction modulo \(p\) of the coefficients. We denote this ring by \(\mathcal{M}(\mathbf{F}_p)\).
    \end{defi}
    
    We want to give a good description of \( \mathcal{M}(\mathbf{F}_p)\). For that purpose, consider the ring homomorphism 
    \begin{equation}
    \label{eq:morphism_mod_forms_mod_p}
        \mathbf{F}_p[e_4,e_6] \rightarrow \mathcal{M}(\mathbf{F}_p), \quad e_i \mapsto \text{ reduction mod } p \text{ of } E_i.
    \end{equation}

    By the discussion in the beginning of the section, this is surjective, thus we may try to describe its kernel.
    In practice, this means finding congruences modulo \(p\) for modular forms. Using congruences for Bernoulli numbers -- the so-called Kummer and von-Staudt Clausen congruences -- Swinnerton-Dyer and Serre (see \cite{swd-modl}, page 22) show that both \( E_{p-1} \) and \(E_{p+1}\) have \(p\)-integral \(q\)-expansion and that
    \begin{align}
    \label{eq:von-staudt-kummer}
    E_{p-1} \equiv_p 1 \, \text{ and } \,  E_{p+1} \equiv_p E_2.
    \end{align}
    We denote \(A\) and \(B\) the unique polynomials in \( \mathbf{Z}_{(p)}[e_4,e_6]\) satisfying
    \begin{align}
        \label{eq:def_AB}
        \begin{split}
        A(E_4,E_6) &= E_{p-1} \\
        B(E_4,E_6) &= E_{p+1}
        \end{split}.
    \end{align}
    
    \begin{prop}[\cite{swd-modl}, Theorem 2]
        \label{prop:swinnerton-dyer}
        For each prime \(p \geq 5\) we have
        \begin{enumerate}
            \item The polynomials \(\tilde A\) and \(\tilde B\) satisfy
              \begin{equation}
                \label{eq:diff_eq_A_B}
                \begin{cases}
                \partial \tilde A = \tilde B \\
                \partial \tilde B = -e_4 \tilde A
                \end{cases},
              \end{equation}
            where \(\partial\) is the vector field of \cref{eq:Ramanujan-Serre_vf};
            \item \(\tilde A\) has no repeated factor and is prime to \(\tilde B\) in \(\overline{\mathbf{F}_p}[e_4,e_6]\);
            \item \(
        \mathcal M(\mathbf{F}_p) \cong \frac{\mathbf{F}_p[e_4,e_6]}{(\tilde A -1)} \).
        \end{enumerate}
    \end{prop}
    \begin{proof}
        The polynomial \(\partial A - B\) is homogeneous of degree \(p+1\). By  \eqref{eq:commutative_diagram/Ram-Serre}
        \[
            f = \partial A(E_4,E_6) - B(E_4,E_6) = 12 \theta E_{p-1} - (p-1) E_2 E_{p-1} - E_{p+1}
        \]
        is a modular form of weight \(p+1\) and -- by \eqref{eq:von-staudt-kummer} -- it reduces to zero modulo \(p\). Therefore, all the coefficients of \(f\) are divisible by \(p\). By the discussion in the beginning of the section, the coefficients of \(\partial A(e_4,e_6) - B(e_4,e_6)\) are also divisible by \(p\). This shows that \(\partial \tilde A = \tilde B\). A similar argument works for the other identity. Item (2) of the lemma will be independently proven in a subsequent section. 

    Next, consider the homomorphism \cref{eq:morphism_mod_forms_mod_p} and denote by \(\mathfrak{a}\) its kernel. Notice that:
    \begin{enumerate}
        \item \((\tilde A-1) \subset \mathfrak a\) by the congruence \(E_{p-1} \equiv_p 1\); 
        \item \(\mathfrak a\) is prime, since \(\mathcal{M}(\mathbf{F}_p) = \mathbf{F}_p[E_4,E_6] \subset \mathbf{F}_p[[q]]\) is an integral domain. 
    \end{enumerate}
    We assert that \(\mathfrak a\) is \emph{not} a maximal ideal. To reach a contradiction, assume that it is. Equivalently, \(\mathbf{F}_p[E_4,E_6] \subset \mathbf{F}_p[[q]]\) is a field, thus every nonzero series \(f(q) \in \mathbf{F}_p[E_4,E_6]\) has an inverse in \(\mathbf{F}_p[E_4,E_6]\). We conclude that \(f(q) = u + O(q)\), where \(u \in \mathbf{F}_p^*\). This is not possible, since for primes \(p \geq 7\) the series \(E_4 -1 = 240 q + O(q^2)\) is not zero and has no constant term. Similarly, for \(p = 5\) the series \(E_6 -1 = -504 q + O(q^2)\) neither has a constant term. We conclude that \(\mathfrak a\) is not maximal. Since \(\mathbf{F}_p[E_4,E_6]\) has dimension \(\leq 2\) and \(\tilde A - 1\) is prime, we must have \(\mathfrak a = (\tilde A - 1)\). In conclusion, 
    the surjective morphism \cref{eq:morphism_mod_forms_mod_p} induces the isomorphism
    \begin{equation*}
        \mathcal M(\mathbf{F}_p) \cong \frac{\mathbf{F}_p[e_4,e_6]}{(\tilde A -1)}.
    \end{equation*}
    \end{proof}

    \begin{lemma}
    \label{lemma:firstintegral}
    The polynomials \(\tilde A\) and \(\tilde B\) satisfy
        \[
            R(\tilde B - e_2 \tilde A) = 0. 
        \]
        Furthermore, if \(X_1, X_2 \in \mathbf{F}_p[e_4,e_6]\) are homogeneous and satisfy \(\deg X_1 \equiv_p -1\), \(\deg X_2 \equiv_p 1\) then \(R(X_2 - e_2 X_1) = 0\) implies that \(\partial X_1 = X_2\) and \(\partial X_2 = - e_4 X_1\).
    \end{lemma}

    \begin{proof}
        This is an application of \cref{eq:ramanujan_vf_2}. Using this, it is easy to see that, if \(\deg X_1 \equiv_p -1\), \(\deg X_2 \equiv_p 1\) then 
        \[
            R(X_2 - e_2 X_1) = 0 \iff \partial X_2 + e_4 X_1 + e_2(X_2 - \partial X_1) = 0.
        \]
        Since \(X_1,X_2\) are independent of \(e_2\), we necessarily have \(\partial X_1 = X_2\) and \(\partial X_2 = - e_4 X_1\).
    \end{proof}
    
    \subsubsection{The polynomials \texorpdfstring{$\tilde A$}{A~} and \texorpdfstring{$\tilde B$}{B~}}
    In this section, we give explicit factorizations of \(\tilde A\) and \(\tilde B\) that work for every prime \(p \geq 5\). The first part of \cref{lemma:factorization} uses the same proof as \cite{swd-modl}. The computation of the exponents \(\delta'\) and \(\epsilon'\) in \cref{tab:delta_epsilon} are new as far as we know. 
    
    Given a nonzero homogeneous polynomial \(P(e_4,e_6) \in \overline {\mathbf{F}}_p [e_4,e_6]\) of degree \(d \geq 0\), we can write uniquely
    \[
        P(e_4,e_6) = \pi_0 e_4^{n_1} e_6^{n_2} \prod_{i = 1}^{n_3} (e_4^3 - \pi_i e_6^2)
    \]
    where \(\pi_0, \pi_1, \ldots, \pi_m \in \overline {\mathbf{F}}_p\) are nonzero although they may repeat \footnote{This is a direct consequence of the fact that a homogeneous polynomial in two variables with coefficients in a algebraically closed field is a product of linear factors}. 
    
    Using this, write 
    \begin{align}
        \label{eq:factorization}
        \begin{split}
        \tilde A(e_4,e_6) &= \alpha_0 e_4^{\delta} e_6^{\epsilon} \prod_{i = 1}^{m} (e_4^3 - \alpha_i e_6^2), \,\,\,\,\,\,\, 12m := p-1 - 4\delta - 6  \epsilon \\
        \tilde B(e_4,e_6) &= \beta_0 e_4^{\delta'} e_6^{\epsilon'} \prod_{j = 1}^{m'} (e_4^3 - \beta_j e_6^2), \,\,\,\,\,\,\, 12m' = p+1 - 4\delta' - 6 \epsilon'  
         \end{split}
    \end{align}
    for \(\alpha_i, \beta_j \in \overline {\mathbf{F}_p}\), \(0 \leq i \leq m\), \(0 \leq j \leq m'\). Throughout the text we will use this same notation, that is
    \begin{align*}
        m (\text{resp. } m') &:= \text{ number of factors } e_4^3 - * e_6^2 \text{ dividing } \tilde A (\text{resp. } \tilde B) \\
        \delta (\text{resp. } \delta') &:= \text{ number of factors } e_4 \text{ dividing } \tilde A (\text{resp. } \tilde B) \\
        \epsilon (\text{resp. } \epsilon') &:= \text{ number of factors } e_6 \text{ dividing } \tilde A (\text{resp. } \tilde B) .
    \end{align*}
    
    \begin{lemma}
        \label{lemma:factorization}
        Consider the unique factorization of \(\tilde A\) and \(\tilde B\) as in \eqref{eq:factorization}. Then, for every prime \(p \geq 5\),
        \begin{enumerate}
            \item \(\tilde A\) has only simple factors in its irreducible decomposition and all are different from those of \(\tilde B\). In particular, \(\tilde A\) and \(\tilde B\) are coprime;
            \item the values of \(\delta, \delta', \epsilon, \epsilon'\) can be computed for each possible class of \(p\) modulo 12 and are given in \cref{tab:delta_epsilon}.
        \end{enumerate}
    \end{lemma}

    \begin{proof}
        First notice that \(\alpha_i, \beta_j \neq 1\) for every \(i,j \geq 1\). This happens because \(E_4^3 - E_6^2\) has no constant term whereas both \(\tilde A(E_4,E_6)\) and \(\tilde B(E_4,E_6)\) have constant term equal to 1. To reach a contradiction, assume that \(\tilde A\) is \emph{exactly divisible} by \((e_4^3 - \alpha e_6^2)^n\) for some \(\alpha \neq 1\) and \(n \geq 2\). This means that \(\tilde A\) factors as \(\tilde A = P (e_4^3 - \alpha e_6^2)^n\) and \(P\) is not divisible by \(e_4^3 - \alpha e_6^2\). Then,
        \[
        \tilde B = \partial \tilde A = \partial P (e_4^3 - \alpha e_6^2)^n + (-12n(1-\alpha)) P e_4^2 e_6 (e_4^3 - \alpha e_6^2)^{n-1} 
        \]
        where \(-12 n (1- \alpha) \neq 0 \) (notice that \(2 \leq n < \frac{p-1}{12}\) so it cannot be a multiple of \(p\)). Therefore,
        \[
            \tilde B  = Q \cdot (e_4^3 - \alpha e_6^2)^{n-1}
        \]
        where \(Q\) is not divisible by \((e_4^3 - \alpha e_6^2)\). Applying \(\partial\) again and using the same argument we get that \(
        - e_4 \tilde A = \partial \tilde B = R \cdot (e_4^3 - \alpha e_6^2)^{n-2}  \) and \(R\) is not divisible by \((e_4^3 - \alpha e_6^2)\), a contradiction.  We conclude that all the \(\alpha_i\) are different. Simple adaptations of this same argument show that
        \begin{enumerate}
            \item \(\alpha_i \neq \beta_j\) for every \(i,j \geq 1\);
            \item \(0 \leq \delta, \epsilon \leq 1\);
            \item \(0 \leq \delta' \leq \delta + 2\) and \(0 \leq \epsilon' \leq \epsilon +1\).
        \end{enumerate}
        
        Knowing that \(\delta,\epsilon \in \{0,1\}\) one can compute from 
        \[
            p-1 = 12m + 4\delta + 6 \epsilon
        \]
        the values of \(\delta,\epsilon\) for each class of \(p \mod 12\). For example, if \(p \equiv_{12} 1\) then \(4 \delta \equiv_6 0 \) and \(6 \epsilon \equiv_4 0\) so the only possibility is \(\delta = 0 = \epsilon\). Repeating the same argument for \(p \equiv_{12} 5,7,11\) fills the two first rows of \cref{tab:delta_epsilon}. For the last two rows we analyze the identities given by the degree of \(\tilde B\)
        \[
            p+1 = 12 m' + 4\delta' + 6 \epsilon'.
        \]
        and use the fact that \(0 \leq \delta' \leq \delta + 2\) and \(0 \leq \epsilon' \leq \epsilon +1\).
    \end{proof} 
    
    \begin{table}[ht!]
    \centering
        \begin{tabular}{|c| c c c c|} 
         \hline
        $p \mod 12$ & 1 & 5 & 7 & 11 \\ [0.5ex] 
         \hline\hline
        $\delta$  & 0 & 1 & 0 & 1 \\ 
        \hline
        $\epsilon$ & 0 & 0 & 1 & 1 \\
        \hline
        $\delta'$  & 2 & 0 & 2 & 0 \\
         \hline
        $\epsilon'$ & 1 & 1 & 0 & 0 \\
         \hline
        \end{tabular}
        \caption{Value of the exponents \(\delta,\delta',\epsilon, \epsilon'\). \label{tab:delta_epsilon}} 
    \end{table}

    \subsection{Computing the \texorpdfstring{$p$}{p}-th power of the Ramanujan vector field}

    \label{sec:computing_the_p-th_power}
    
    Recall that the Ramanujan vector field is the derivation
    \[
        R:= \frac{e_2^2 - e_4}{12} 
   	\frac{\partial}{\partial e_2} + 
   	\frac{e_2e_4 - e_6}{3}\frac{\partial}{\partial e_4} + 
	\frac{e_2 e_6 - e_4^2}{2} 
	\frac{\partial}{\partial e_6}    
    \]
    and that
    \begin{align*}
        F &:= - 12 \frac{\partial}{\partial e_2}, \\
        H &:= 2 e_2 \frac{\partial}{\partial e_2} + 4 e_4
			\frac{\partial}{\partial e_4} + 6 e_6 
			\frac{\partial}{\partial e_6}.
    \end{align*}
    These vector fields are defined a priori in the 3-dimensional affine space \(\mathbf{A}^3_{\mathbf{Z}[1/6]}\). As mentioned in the introduction (see \cref{eq:U}), we write 
    \[
        U := \Spec \mathbf{Z}[1/6, e_2,e_4,e_6, (e_4^3 - e_6^2)^{-1}] \subset \mathbf{A}^3_{\mathbf{Z}[1/6]}.
    \]
    Following the convention of modern algebraic geometry, given \(S\) a \(\mathbf{Z}[1/6]\)-algebra, \(U_S\) denotes \(U \otimes S = \Spec S[e_2,e_4,e_6, (e_4^3 - e_6^2)^{-1}]\). The same is true for \(\mathbf{A}^3_S\). 
        
    \begin{lemma}
        \label{lemma:lie_algebra}
        For every \(\mathbf{Z}[1/6]\)-algebra \(S\), the following are true:
        \begin{enumerate}
            \item \(\mathcal{T}_{U_S}\) is locally free of rank \(3\) and it is globally generated by \(R,F,H\);
            \item The derivations \(R,F,H\) are homogeneous of degrees \( 2,-2, 0\) respectively. 
        \item The derivations \(R,F,H\) satisfy the following commutation relations:
        \begin{align}
            \label{eq:bracketrel}
            \begin{split}
                [R,F] &= H \\
                [H,F] &= -2F \\
                [H,R] &= 2R.
            \end{split}
        \end{align}
        \end{enumerate}
    \end{lemma}

    \begin{proof}
        The only nontrivial item is the first one. We first observe that the partial derivatives \(\partial/\partial e_i\) define an independent 
		basis for the derivations, so this module is free of rank 3. It remains to show that \(R,F\) and \(H\) are linearly independent. 
		For each \(i = 1,2,3\), let \(P_i \in S[e_2,e_4,e_6,(e_4^3 - e_6^2)^{-1}]\) 
		be such that 
		\[
			P_1(e_2,e_4,e_6)  \cdot R + P_2(e_2,e_4,e_6) \cdot  F + P_3(e_2,e_4,e_6) \cdot H = 0.
		\]
		We obtain the following system 
		\[
		\begin{pmatrix}
		    	\frac{e_2^2 - e_4}{12}  && - 12 && 2 e_2  \\
		      \frac{e_2 e_4 - e_6}{3}  && 0 &&  4 e_4  \\
		     \frac{e_2e_6 - e_4^2}{2}  && 0 && 6 e_6 
		\end{pmatrix}
		\begin{pmatrix}
		    	P_1 \\
		    	P_2 \\
		    	P_3 
		\end{pmatrix}
		= 0
		\]
		which has determinant \( 24 (e_4^3 - e_6^2) \). We conclude that the determinant of the system is invertible in 
		\(S[e_2,e_4,e_6,(e_4^3 - e_6^2)^{-1}]\), so it has no non-trivial solutions. 
    \end{proof}

     \begin{proof}[Proof of \cref{thm:main_theorem}.]
    	Let \(r_1, r_2, r_3 \in k[e_2,e_4,e_6, (e_4^3-e_6^2)^{-1}]\) be the coefficients of 
    	\(R^p\) in the basis \(R,F,H\) of \(H^0(U_k, \mathcal{T}_{U_k})\).
        This is possible by \cref{lemma:lie_algebra}. We want to show that \(r_1,r_2,r_3\) are the coefficients \(
       \tilde{A}^2, - \left(\frac{\tilde B - e_2 \tilde A}{12}\right)^2, \tilde{A} \left(\frac{\tilde B - e_2 \tilde A}{12} \right)
       \), respectively. In other words, we want to show that 
       \begin{equation}
           \label{eq:p-th_power_2}
           R^p = \tilde{A}^2 \cdot R - \left(\frac{\tilde B - e_2 \tilde A}{12}\right)^2 \cdot F + \tilde{A} \left(\frac{\tilde B - e_2 \tilde A}{12} \right) \cdot H.
       \end{equation}
       It is easy to see that this is equivalent to \cref{eq:p-th_power}.
       
       Firstly, we claim that it suffices to prove the identity \(r_1 = \tilde{A}^2\). 
        Notice that \(R^p\) commutes with \(R\), so we have an identity of the form 
        \[
            [R,R^p] = [R, r_1 R + r_2 F + r_3 H] = 0.
        \]
        This gives by the third part of \cref{lemma:lie_algebra} that the coefficients satisfy the differential system
        \begin{equation}
        \label{eq:diffequation}
            \begin{cases} 
                R(r_1) = 2r_3  \\ 
                R(r_2) = 0 \\ 
                R(r_3) = -r_2 
            \end{cases}.
        \end{equation}
        Now we observe that a solution \((r_1,r_2,r_3)\) to the system \eqref{eq:diffequation} is uniquely determined by the first entry as \(r_3 = R(r_1)/2\) and \(r_2 = - R(r_3)\).
       On the other hand, \cref{lemma:firstintegral} implies that the triple
       \[
       \left(\tilde{A}^2, - \left(\frac{\tilde B - e_2 \tilde A}{12}\right)^2, \tilde{A} \left(\frac{\tilde B - e_2 \tilde A}{12} \right)\right)
       \]
       forms a solution to \eqref{eq:diffequation}. Finally, assuming that \(r_1 = \tilde{A}^2\), we conclude that \(r_2 = - \left(\frac{\tilde B - e_2 \tilde A}{12}\right)^2\) and \(r_3 =  \tilde{A} \left(\frac{\tilde B - e_2 \tilde A}{12} \right)\).

       The rest of the proof is showing that \(r_1 = \tilde{A}^2\). Consider the morphism 
       \begin{equation}
       \label{eq:formal_solution}
           \varphi: k[e_2,e_4,e_6,(e_4^3 -e_6^2)^{-1}]  \to k((q))
       \end{equation}
       given by \(e_i \mapsto E_i\).  As discussed in \cref{subsec:Ramanujan-Serre}, \(\varphi\) defines a formal solution to the Ramanujan vector field, meaning that 
       \begin{equation}
            \label{eq:solution}
            \varphi \circ R (P) = \theta \circ \varphi(P)
       \end{equation}
       for any \(P\). We make three claims:
    	\begin{enumerate}
    		\item \(r_1\) is homogeneous of degree \(2p-2\) and does not depend on \(e_2\);
    		\item The only homogeneous element \( f \in k[e_4,e_6,(e_4^3-e_6^2)^{-1}]\) of degree \(2p-2\) such that \(\varphi(f) = 1\) is \(\tilde{A}^2\);
    		\item \(\varphi(r_1) = 1\).
    	\end{enumerate}
        These three claims imply that \( r_1 = \tilde A^2\).
  
  		For the first step, it is easy to see that \(R^p\) is homogeneous of degree \(2p\). Analyzing the degrees of \(r_1,r_2,\) and \(r_3\) in the identity \(R^p = r_1 R + r_2 F + r_3 H\) and using \cref{lemma:lie_algebra}  we conclude that 
        \(r_1, r_2,\) and \(r_3\) are homogeneous of degrees \(2p - 2\), \(2p + 2\), and \(2p\) respectively. To show that \(r_1\) does not depend on \(e_2\) consider the derivation
        \([F,R^p] \). Since \(H^0(U_k, \mathcal{T}_{U_k})\) is a restricted Lie algebra of characteristic \(p\) (see \cite{jacobson_lie_algebras}, Chapter V, Section 7), we have 
        \[
          [F,R^p] = [\ldots[[F,R], R], \ldots R]
        \]
        where, in the right hand side, we take the bracket with \(R\) a total of \(p\) times.  Using \eqref{eq:bracketrel}, we have
        \[
            [[[F,R],R],R] = [[-H,R]R] = [-2R,R] = 0
        \]
        and, as \(p \geq 5\), \([F,R^p] =0\). From this and from \(R^p = r_1 R + r_2 F+ r_3 H\) it is easy to show that \(F(r_1) = 0\).

    	For the second claim, assume that \(\varphi(f) =1\). Then \(\varphi(f-1) = 0\) and from \cref{sec:mod_forms_mod_p} this happens if, and only if,  \(f-1 \in \langle \tilde{A}-1 \rangle\). Thus there exists some \(D\) such that \(f-1 = D (\tilde{A}-1)\). Using that \(\tilde{A}\) is homogeneous of degree \(p-1\) and decomposing \(D\) in its homogeneous factors \(D = \sum_{d \geq s} D_d\) we can analyze every homogeneous component of the identity \(f-1 = D(\tilde A-1)\) to conclude
    	\[
    		\begin{cases}
            D_0 = 1 \\
    		D_{p-1} = \tilde{A} \\
    		D_{kp-k} = \tilde{A}^{k-2} (f - \tilde{A}^2), &\text{ if } k>1 \\
    		D_d = 0, &\text{ if } kp - k < d < (k+1)p - (k+1) \text{ and } d< 0.\\
    		\end{cases}	
    	\]
        But \(D\) has finite degree, so \(D_{kp-k} = 0\) for some \(k >1\), which implies that \(f - \tilde{A}^2 = 0\). 
    	
    	We finish by showing that \(r_1\) satisfies \(\varphi(r_1) = 1\). First notice that 
        \[
            \theta^p = \theta 
        \]
        as a derivation in \(k((q))\), since 
        \(
            \theta^p\left(\sum_{n \geq s} a_n q^n\right) = \sum_{n \geq s} n^p a_n q^n 
        \)
        and \(n^p \equiv_p n \) (Fermat's Little Theorem). By \cref{eq:solution} we have \(\theta^p \circ \varphi = \varphi \circ R^p \), therefore
    	\(
    		\varphi (R^p- R) = 0.
    	\)
    	  Consequently, \((R^p - R) (f) \in \ker(\varphi)\) for every \(f\). Write
    	\[
    		R^p - R = (r_1 -1)R + r_2 F+  r_3 H
    	\]
    	and let \(P := \frac{\tilde B - e_2 \tilde A}{12}\). Then \(R(P) = 0\) (by \cref{lemma:firstintegral}), \(F(P) = \tilde{A}\) and \(H(P) =  P\). We conclude that 
    	\[
    		(R^p- R)(P) = r_2 \tilde{A} + r_3P.
    	\]
    	Since \(\varphi(P) = 0\) and \(\varphi(\tilde{A}) = 1\) we get \( \varphi(r_2) = 0\). Similarly, we compute \((R^p-R)(\tilde{A}^2)\) to get 
    	\[
    		(R^p - R)(\tilde{A}^2) = 2 (r_1-1) \tilde{A} P-2 r_3 \tilde{A}^2 
    	\]
    	which has to be zero when we apply \(\varphi\), so \(\varphi(r_3) = 0\). Now, taking any \(f\) such that \((\varphi \circ R)(f) \neq 0\) (for example, take \(f = e_4^3 - e_6^2\)) we get
    	\((R^p - R)(f) = (r_1 -1) R(f) + r_2 F(f) + r_3 H(f) \). Applying \(\varphi\) we get zero, but \(\varphi(r_2) = \varphi(r_3) = 0\) already and \(\varphi \circ R(f) \neq 0\), so we conclude that \(\varphi(r_1 -1) = 0\). 
        
        Lastly, it is clear that the formula obtained is valid for the whole affine space \(\mathbf{A}^3_k\), since the coefficients \(r_i\) for \(i = 1,2,3\) have no poles in the surface \(\{e_4^3 - e_6^2 = 0\}\).
        \end{proof}
    
    \begin{cor}
    \label{cor: Sheperd-Barron-Ekedahl}
    Let \(k\) be a field of characteristic \(p \geq 5\). The \(p\)-th power of the Ramanujan vector field \(R^p \in H^0(U_k, \mathcal{T}_{U_k})\) is not a \(H^0(U_k,\mathcal{O}_{U_k})\)-multiple of \(R\). In other words, the \(p\)-curvature of the foliation generated by \(R\) in \(U_k\) is different from zero.
    \end{cor}

    \begin{proof}
        By \cref{eq:p-th_power_2}, \(R^p\) is a multiple of \(R\) only when \(\tilde B = e_2 \tilde A\). Since \(\tilde B\) is independent of \(e_2\), this is only possible when both polynomials are identically zero, but this is false for every prime \(p \geq 5\).
    \end{proof}

\subsubsection{Singular set of \texorpdfstring{$R^p$}{Rp}}
\label{sec:singular_set}
    
    We describe the singular set of \(R^p\), i.e., the set where \(R^p\) vanishes. 

    \begin{cor}
        Let \(k\) be an algebraically closed field of characteristic \(p \geq 5\). Let \(\Sing(R^p)\) be the singular set of \(R^p\) in \(\mathbf{A}^3_k\), i.e., the algebraic set defined by the ideal \(\langle R^p(e_i): i = 2,4,6\rangle \). Then 
        \[
            \Sing(R^p) = \{e_4^3 - e_6^2 = 0\}.
        \]
    \end{cor}

    \begin{proof}
            Using \cref{thm:main_theorem} it is easy to show that 
        \begin{align*}
            R^p(e_2) &= \frac{1}{12}\left( \tilde B^2 - e_4 \tilde A^2 \right), \\
            R^p(e_4) &= \frac{\tilde A}{3} \left( e_4 \tilde B  - e_6 \tilde A \right), \\
            R^p(e_6) &= \frac{\tilde A}{2} \left( e_6 \tilde B - e_4^2 \tilde A \right).
        \end{align*}
        For a moment, forget that we are taking modulo \(p\) of the coefficients and consider the expressions in the right hand side taking \(A\) and \(B\) as the original polynomials with coefficients in \(\mathbf{Z}_{(p)}\). Substituting \(e_i\) by \(E_i\) we see that the expressions represent cusp forms of weights \(2p+2, 2p+4\) and \(2p+6\) respectively. Since \(p \geq 5\), they are necessarily divisible by \(\Delta\). This shows that \(\langle R^p(e_i): i = 2,4,6 \rangle \subset \langle e_4^3 - e_6^2 \rangle\), therefore, \(\Sing(R^p) \supset \{e_4^3 - e_6^2 = 0\}\). For the other inclusion, pick any point \(x = (x_1,x_2,x_3) \in \Sing(R^p)\). Then either \(\tilde A(x) = 0\) or \(\tilde A(x) \neq 0\). In the first case, by the first equation we have \(\tilde B(x) = 0\), so 
        \[
            x \in \{\tilde A = 0\} \cap \{\tilde B = 0\} = \{e_4 = e_6 = 0\} \subset \{e_4^3 - e_6^2 =0\}.
        \]
        In the case \(\tilde A(x) \neq 0\), by the first equation 
        \[
            x_2  = \frac{\tilde B^2(x)}{ \tilde A^2(x)}
        \]
        and by the second one 
        \[
            x_3 = \frac{\tilde B^3(x)}{\tilde A^3(x)}
        \]
        so \(x_2^3 - x_3^2 = 0\).
    \end{proof}

%% file: supersingular_locus.tex
\section{Supersingular locus}
    We use one of the descriptions of the supersingular polynomial in terms of the polynomial \(\tilde A\) to compute equations for the locus of supersingular elliptic curves in \(U_k\) and we prove \cref{thm:supersingular_locus}.
    
    \subsection{Parametrizing the moduli space}
    \label{subsec:moduli_space}
    It is known (see \cite{movasati-2012}, \cite{tiago_higher_ramanujan_equations}) that \(U\) (\cref{eq:U}) represents the moduli functor 
    \begin{equation}
        \mathcal{B}: \label{eq:moduli_functor}
	   \text{Sch}/\mathbf{Z}[1/6] \longrightarrow \text{Sets}
    \end{equation}
    taking the \(\mathbf{Z}[1/6]\)-scheme \(S\) to the set of equivalence classes of triples \((E/S, \omega, \alpha)\), where \(E/S\) is a family of elliptic curves, \(\omega \in F^1 H^1_{DR}(E/S)\) and \(\alpha \in H^1_{DR}(E/S) \) form a basis of \(H^1_{DR}(E/S) \) satisfying \(\langle \omega, \alpha \rangle =1\). Moreover, the moduli problem has a universal family \(\mathcal{E}/U\) where \(\mathcal{E}\) has an affine equation
    \begin{equation}
        \label{eq:univ_family}
        \mathcal{E}: y^2 = 4 \left( x + \frac{e_2}{12}\right)^3 - \frac{e_4}{12} \left(x + \frac{e_2}{12}\right) + \frac{e_6}{216}.
    \end{equation}
    
    Therefore, given an algebraically closed field \(k\) of characteristic \(p \geq 5\), the closed points of \(U_k\) are in bijection with the isomorphism classes of triples \((E/k, \omega, \alpha)\). Since we have the universal family \(\mathcal{E}_k/U_k\) as in \cref{eq:univ_family}, the bijection is given by 
    \begin{align}
        \label{eq:bijection_moduli}
        \begin{split}
        U_k(k) = \mathbf{A}^3_k(k) \setminus \{e_4^3 - e_6^2 = 0\} &\to \mathcal{B}(k)  \\
        (a,b,c) &\mapsto \text{class of } (E_{a,b,c}, dx/y, x dx/y)
        \end{split}
    \end{align}
    where the elliptic curve \(E_{a,b,c}/k\) has affine coordinates \(E_{a,b,c}: y^2 = 4 \left( x + \frac{a}{12} \right)^3 - \frac{b}{12} \left( x + \frac{a}{12} \right) + \frac{c}{216}\). 

    Consider the map given by the \(j\)-invariant \(\mathcal{B}(k) \to \mathbf{A}^1_k(k)\). Then we can lift it to \(U_k(k)\) to get a map given (in \(k\)-points) by
    \[
	   (a,b,c) \mapsto j(E_{a,b,c}) = 1728 \frac{b^3}{b^3 - c^2}.
    \]
    This map is induced by the morphism of schemes 
    \begin{equation} 
	   \label{J-map}
	   J: U_k \to \mathbf{A}^1_k 
    \end{equation}
    associated to the ring homomorphism
    \begin{align*}
	   k[t] &\longrightarrow k[e_2,e_4,e_6,(e_4^3 - e_6^2)^{-1}] \\
	   t &\mapsto 1728 \frac{e_4^3}{e_4^3 - e_6^2}.
    \end{align*}

    \subsection{Computing the supersingular locus}    
    The \(p\)-th supersingular polynomial is defined as the single variable polynomial given by 
    \[
	   ss_p(t) := \prod_{E/k \text{ supersingular}} (t - j(E))
    \]
    where the product runs through isomorphism classes of supersingular elliptic curves over \(k\).
    It is well defined, since there are only finitely many supersingular \(j\)-values over \(k\). We want to describe the locus of supersingular elliptic curves of \( U_k\), that is, the preimage of the zeroes of \(ss_p(t)\) by the morphism \(J: U_k \to \mathbf{A}^1_k\). 

    Let \(E_{p-1}\) be the \(p-1\)-th Eisenstein series and write 
    \begin{equation}
    \label{eq:E_{p-1}}
       E_{p-1} = (E_4^3 - E_6^2)^m E_4^{\delta} E_6^{\epsilon} f(j), 
    \end{equation}
    where \(p-1 = 12m + 4\delta + 6\epsilon\) as in \cref{eq:factorization} and \(j = j(\tau) = 1728 E_4^3/(E_4^3 - E_6^2)\). Here, \(f(j)\) is a polynomial in \(\mathbf{Q}[j]\) with \(p\)-integral coefficients and by \cite{kaneko-zagier-supersingular_j-invariants} we have
    \begin{equation}
	   \label{supersingular_polynomial}
	   ss_p(t)  \equiv_p \pm t^{\delta} (t - 1728)^{\epsilon} f(t).
    \end{equation}

    The equation \cref{eq:E_{p-1}} defines the identity in \(\mathbf{Q}(e_4,e_6)\)
    \[
    	A(e_4,e_6) = (e_4^3 - e_6^2)^m e_4^{\delta} e_6^{\epsilon} f\left( 1728 \frac{ e_4^3}{e_4^3 - e_6^2} \right).
    \]
    Multiplying both sides by \(\left( 1728 \frac{ e_4^3}{e_4^3 - e_6^2} \right)^{\delta}\left( 1728 \frac{ e_4^3}{e_4^3 - e_6^2} -1728\right)^{\epsilon}\) and reducing modulo \(p\) we get from \cref{supersingular_polynomial}
    \begin{equation}
	   \label{supersingular_polynomial_2}
	   ss_p\left( 1728 \frac{ e_4^3}{e_4^3 - e_6^2} \right) \equiv_p c \frac{e_4^{2 \delta} e_6^{\epsilon} A(e_4,e_6)}{(e_4^3 - e_6^2)^{m+\delta+\epsilon}} 
    \end{equation}
    where \(c\) is a nonzero constant in \(k\).

    \begin{prop}
        \label{prop:unreduced_scheme}
	   Let \(p \geq 5\) be a prime number and 
	   \[
		J: U_k \to \mathbf{A}^1_k	
	   \]
	   the morphism induced by the \(j\)-invariant. 
	   The scheme \(J^{-1}(V(ss_p(t))) \hookrightarrow U_k\) can be written as the closed subscheme defined as the 
	   zero set of \(f_p(e_2,e_4,e_6) \in k[e_2,e_4,e_6] \) where 
	   \[ 
	   f_p(t_1,t_2,t_3) =
	   \begin{cases} 
		\tilde{A}(e_4,e_6), & \text{ if } p \equiv_{12} 1 \\
		e_4^2 \tilde{A}(e_4,e_6), & \text{ if } p \equiv_{12} 5\\
		e_6 \tilde{A}(e_4,e_6), & \text{ if } p \equiv_{12} 7 \\
		e_4^2 e_6 \tilde{A}(e_4,e_6), & \text{ if } p \equiv_{12} 11
	   \end{cases}.
	   \]
    \end{prop}

    \begin{proof}
	   The scheme \( J^{-1}(V(ss_p)) \) is given by
	   \[
	       V(ss_p) \times_{\mathbf{A}^1_k} U_k  \hookrightarrow U_k
	   \]
	   where the \(\mathbf{A}^1_k\)-structure of \(U_k\) is given by \eqref{J-map}.
	   After convenient identifications this is the morphism of schemes induced by the \(k[t]\)-homomorphism
	   \begin{align*}
		  k[e_2,e_4,e_6,(e_4^3 - e_6^2)^{-1}] \twoheadrightarrow \frac{k[e_2,e_4,e_6,(e_4^3 - e_6^2)^{-1}]}{(ss_p(1728 e_4^3/(e_4^3 - e_6^2)))}.
	   \end{align*}
	   Since \(p\) is prime and different from \(2\) and \(3\) then \(p\) is either \(1,5,7\) or \(11\) modulo 12. For each of these cases we find the value of \(\delta\) and \(\epsilon\) as in \cref{tab:delta_epsilon}. Finally, using 
	   \eqref{supersingular_polynomial_2} and getting rid of unities we get the expressions for \(f_p(e_2,e_4,e_6)\).
    \end{proof} 

    \begin{proof}[Proof of \cref{thm:supersingular_locus}]    

    We start by showing that \(J^{-1}(V(ss_p))_{\text{red}} = V(\tilde A)\) for every prime \(p \geq 5\). Indeed, by \cref{eq:factorization}
    \begin{equation*}
    	\tilde A(e_4,e_6) = \alpha_0 \prod_{i = 1}^m (e_4^3 - \alpha_i e_6^2) e_4^{\delta} e_6^{\epsilon}
    \end{equation*}
    with \(\delta,\epsilon\) as in \cref{tab:delta_epsilon}. Additionally, the numbers \(\alpha_1,\ldots, \alpha_m \in k\) are all different and non-zero. Analyzing for each \(l \in \{5,7,11\}\), if \(p \equiv_{12} l\) there are extra powers of \(e_4\) and \(e_6\) in \(f_p(e_4,e_6)\). Getting rid of the extra powers we end up with \(\tilde A(e_4,e_6)\) which defines a reduced scheme. Moreover, by \eqref{eq:factorization} it is easy to see that each component of \(SS_p\) is smooth.

    Now we show that the Ramanujan vector field is transversal to the supersingular locus \(SS_p\) at every point. Recall from \cref{lemma:firstintegral} that
	\[
		R(\tilde A) = d\tilde A (R) = \frac{1}{12}(\tilde B - e_2 \tilde A).
	\]
	Using the fact that \(\tilde B\) is coprime to \(\tilde A\), the hypersurface \(\{\tilde B = 0\}\) has no common components with \(SS_p\). Thus, if \(x\) is a 
	closed point of \(SS_p\) then \(\tilde B(x) \neq 0\). We conclude that \(R_x \notin \ker (dA)_x = T_x SS_p\). Since each component of \(SS_p\) is smooth, \(T_xSS_p\) is a dimension 2 subspace of \(T_x\mathbf{A}^3_k\), therefore \(k R_x + T_xSS_p = T_x \mathbf{A}^3_k\).

    Finally, we show that the supersingular locus is the reduced scheme associated to the zeroes of \(R^p + \left(\frac{\widetilde B}{12}\right)^2 F \in \mathcal{T}_{U_k}\). Computing \(R^p + \left(\frac{\widetilde B}{12}\right)^2 F\) in the basis \(\partial/\partial e_i\), \(i = 2,4,6\) we get
    \begin{align*}
            R^p(e_2) + \left(\frac{\widetilde B}{12}\right)^2 F(e_2) &= - \frac{1}{12}  e_4 \widetilde A^2,   \\
            R^p(e_4) + \left(\frac{\widetilde B}{12}\right)^2 F(e_4) &= \frac{\tilde A}{3} \left( e_4 \tilde B  - e_6 \tilde A \right), \\
            R^p(e_6) + \left(\frac{\widetilde B}{12}\right)^2 F(e_6)&= \frac{\tilde A}{2} \left( e_6 \tilde B - e_4^2 \tilde A \right). 
        \end{align*}
        By the equations it is clear that the ideal of \(\left\{R^p + \left(\frac{\widetilde B}{12}\right)^2 F = 0\right\}\) is contained in \(\langle \widetilde A\rangle\). On the other hand, multiplying the second equation by \(6 e_6\) and the third by \(-4e_4\) and summing them up we get \(2\widetilde{A}^2(e_4^3 - e_6^2) \), which means that the radical of the ideal of \( \left\{R^p + \left(\frac{\widetilde B}{12}\right)^2 F = 0\right\} \) is equal to \(\langle \widetilde A \rangle\).
        \end{proof}